\documentclass[reqno]{amsart}

\usepackage{amssymb,amsfonts,amsmath,mathdots,amsthm}
\usepackage{float}

\usepackage{tikz-cd}
\usepackage{lmodern}
\usepackage{mathrsfs}
\usepackage[sc]{mathpazo}
\DeclareMathAlphabet{\mathfrak}{U}{euf}{m}{n}
\usepackage{mathtools}  
\usepackage{calligra}
\usepackage{xcolor}
\usepackage[hidelinks]{hyperref}
\usepackage{accents}
\usepackage{subcaption}
\usepackage{dashbox}
\usepackage{enumitem}
\usepackage{tabu}

\usepackage{hyperref}

\hypersetup{colorlinks=true,
	linkcolor=magenta,
	citecolor=magenta}

\setlist[itemize]{leftmargin=*}

\usepackage{relsize}
\usepackage{etoolbox,fancyhdr}
\usepackage{indentfirst}  

\usetikzlibrary{arrows,arrows.meta,shapes.misc,decorations.markings,positioning,patterns}

\let\oldproofname=\proofname
\renewcommand{\proofname}{\bfseries\textup{\oldproofname}}

\mathtoolsset{showonlyrefs} 

\DeclareMathOperator{\Z}{\mathbb{Z}}
\DeclareMathOperator{\F}{\mathbb{F}}

\DeclareMathOperator{\colimit}{\textup{colim}}

\theoremstyle{plain}

\newtheorem{prop}{Proposition}[section]

\makeatother
\title{The $C_{2^n}$ Borel dual Steenrod algebra}
\author{Nick Georgakopoulos}

\begin{document}

	\begin{abstract}In this very short note, we expand the Hu-Kriz computation of the $G$-equivariant Borel dual Steenrod algebra in characteristic $2$, from the group $G=C_2$ to all power-$2$ cyclic groups $G=C_{2^n}$.
	\end{abstract}

	\maketitle{}
	
	\setcounter{tocdepth}{1}
	\tableofcontents
	
	\section{Introduction}\label{Intro}	
	
	In this companion piece to \cite{BC4S2}, we show that the $C_2$-equivariant Borel dual Steenrod algebra computation in \cite{HK96} generalizes to all groups $G=C_{2^n}$. More precisely, we give an explicit description of the $RO(C_{2^n})$-graded ring of the homotopy fixed points $(H\F_2\wedge H\F_2)_{\bigstar}^{hC_{2^n}}$ as a Hopf algebroid over $(H\F_2)^{hC_{2^n}}_{\bigstar}$, where $\F_2$ stands for the constant $C_{2^n}$-Green functor associated to the field of two elements. We also compare our description to the dual description of the Borel Steenrod algebra of \cite{Gre88}.

	\subsection*{Acknowledgment}We would like to thank Peter May for his numerous editing suggestions, including the idea to split off this paper from \cite{BC4S2}.
	
	\section{Conventions and notations}\label{Notations}
	
We will use the letter $k$ to denote the field $\F_2$ with trivial $G=C_{2^n}$ action, the constant $G$-Mackey functor $k=\underline{\F_2}$ and the corresponding equivariant Eilenberg-MacLane spectrum $Hk$. The meaning should always be clear from the context.\medbreak

Henceforth all our co/homology will be in $k$ coefficients. We use $k_{\bigstar}(X)$ to denote the $RO(G)$-graded Mackey functor of $G$-equivariant homology in $k$-coefficients. The value of $k_{\bigstar}(X)$ on the $G/H$ orbit is denoted by  $k_{\bigstar}^H(X)$. \medbreak
	
The real representation ring $RO(C_{2^n})$ is spanned by the irreducible representations $1,\sigma, \lambda_{s,k}$ where $\sigma$ is the $1$-dimensional sign representation and $\lambda_{s,m}$ is the $2$-dimensional representation given by rotation by $2\pi s (m/2^n)$ degrees for $1\le m$ dividing $2^{n-2}$ and odd $1\le s<2^n/m$. Note that $2$-locally, $S^{\lambda_{s,m}}\simeq S^{\lambda_{1,m}}$ as $C_{2^n}$-equivariant spaces, by the $s$-power map. Therefore, to compute $k_{\bigstar}(X)$ for $\bigstar \in RO(C_{2^n})$ it suffices to only consider $\bigstar$ in the span of $1,\sigma,\lambda_k:=\lambda_{1,2^k}$ for $0\le k\le n-2$ ($\lambda_{n-1}=2\sigma$ and $\lambda_n=2$).\medbreak

For $V=\sigma$ or $V=\lambda_m$, denote by $a_V\in k_{-V}^{C_{2^n}}$ the Euler class induced by the inclusion of north and south poles $S^0\hookrightarrow S^V$; also denote by $u_V\in k_{|V|-V}^{C_{2^n}}$ the orientation class generating the Mackey functor  $k_{|V|-V}=k$ (\cite{HHR16}).

\section{Borel cohomology}

Let $EG$ be a contractible free $G$-space and $\tilde EG$ be the cofiber of the collapse map $EG_+\to S^0$.	For a spectrum $X$ we use the notation $X_h=EG_+\wedge X$, $X^h=F(EG_+,X)$ and $X^t=\tilde EG\wedge X^h$; there is a cofiber sequence
$$X_h\to X^h\to X^t$$
The $G$-fixed points of $X_h,X^h,X^t$ are the nonequivariant spectra of homotopy orbits $X_{hG}$, homotopy fixed points $X^{hG}$ and Tate fixed points $X^{tG}$ respectively.

The orientation classes $u_V:k\wedge S^{|V|}\to k\wedge S^{V}$ are nonequivariant equivalences, hence induce $G$-equivalences in $X_h,X^h,X^t$ for a $k$-module $X$, so they act invertibly on $X_{h\bigstar}, X^h_{\bigstar}$ and $X^{t}_{\bigstar}$. This implies that
	\begin{equation}
	X_{h\bigstar}\approx X_{h|\bigstar|}\text{ , }X^h_{\bigstar}=X^h_{|\bigstar|}\text{ , }X^t_{\bigstar}=X^t_{|\bigstar|}
	\end{equation}
	and the $RO(G)$ graded part is determined by the integer graded part.	
	
	\begin{prop}For $G=C_{2^n}$ and $n>1$:
	 \begin{gather}
	 k^{hG}_{\bigstar}=k[a_{\sigma},a_{\lambda_0}, u_{\sigma}^{\pm}, u_{\lambda_0}^{\pm},...,u_{\lambda_{n-2}}^{\pm}]/a_{\sigma}^2\\	
	 k^{tG}_{\bigstar}=k[a_{\sigma},a_{\lambda_0}^{\pm}, u_{\sigma}^{\pm}, u_{\lambda_0}^{\pm},...,u_{\lambda_{n-2}}^{\pm}]/a_{\sigma}^2
	 \end{gather}
	 and $k_{hG\bigstar}=\Sigma^{-1}k^{tG}_{\bigstar}/k^{hG}_{\bigstar}$ (forgetting the ring structure). The map $k_{hG\bigstar}\to  k^{hG}_{\bigstar}$ is trivial.
	 \end{prop}
 
 \begin{proof}The homotopy fixed point spectral sequence becomes:
 	\begin{equation}
 		H^*(G;k)[u_{\sigma}^{\pm}, u_{\lambda_0}^{\pm},...,u_{\lambda_{n-2}}^{\pm}]\implies k^{hG}_{\bigstar}
 	\end{equation}
 We have $H^*(G;k)=k^*BG=k[a]/a^2\otimes k[b]$ where $|a|=1$ and $|b|=2$. The spectral sequence collapses with no extensions and we can identify $a=a_{\sigma}u_{\sigma}^{-1}$ and $b=a_{\lambda_0}u_{\lambda_0}^{-1}$. Finally, $\tilde EG=S^{\infty \lambda_0}=\colimit (S^{\lambda_0}\xrightarrow{a_{\lambda_0}}S^{\lambda_0}\xrightarrow{a_{\lambda_0}}\cdots)$ so to get $k^{tG}_{\bigstar}$ we are additionally inverting $a_{\lambda_0}$.
 
 \end{proof}
	 For $G=C_2$ we have
	 	 \begin{gather}
	 	k^{hC_2}_{\bigstar}=k[a_{\sigma}, u_{\sigma}^{\pm}]\\	
	 	k^{tC_2}_{\bigstar}=k[a_{\sigma}^{\pm}, u_{\sigma}^{\pm}]
	 \end{gather}
	 and $k_{hC_2\bigstar}=\Sigma^{-1}k^{tC_2}_{\bigstar}/k^{hC_2}_{\bigstar}$ (forgetting the ring structure). The map $k_{hC_2\bigstar}\to  k^{hC_2}_{\bigstar}$ is trivial.
	 
	 \section{The Borel dual Steenrod algebra}\label{Borel}
	 
	 The $G$-Borel dual Steenrod algebra is
	 \begin{equation}
	 (k\wedge k)_{\bigstar}^{hG}
	 \end{equation}
	 This is a Hopf algebroid over $k_{\bigstar}^{hG}$.
	 
	 We will implicitly be completing it at the ideal generated by $a_{\sigma}$ for $G=C_2$, and at the ideal generated by $a_{\lambda_0}$ for $G=C_{2^n}$, $n>1$ (see \cite{HK96} pg. 373 for more details in the case of $G=C_2$). With this convention, Hu-Kriz computed the $C_2$-Borel dual Steenrod algebra to be
	 \begin{equation}
	 (k\wedge k)_{\bigstar}^{hC_2}=k_{\bigstar}^{hC_2}[\xi_i]
	 \end{equation}
	 for $|\xi_i|=2^i-1$ ($\xi_0=1$). The generators $\xi_i$ restrict to the Milnor generators in the nonequivariant dual Steenrod algebra and 
	 \begin{gather}
	 		 \Delta(\xi_i)=\sum_{j+k=i}\xi_j^{2^k}\otimes \xi_k\\
	 	\epsilon(\xi_i)=0\text{ , }i\ge 1\\
	 \eta_R(a_{\sigma})=a_{\sigma}\\
	 \eta_R(u_{\sigma})^{-1}=\sum_{i=0}^{\infty}a^{2^i-1}_{\sigma}u^{-2^i}_{\sigma}\xi_i
	 \end{gather}

	 \begin{prop}\label{Power2Borel}For $G=C_{2^n}$, $n>1$, 
	 	 \begin{equation}
	 (k\wedge k)^{hG}_{\bigstar}=k^{hG}_{\bigstar}[\xi_i]
	 \end{equation}
	 for $|\xi_i|=2^i-1$ restricting to the $C_{2^{n-1}}$ generators $\xi_i$, with
	 \begin{gather}
	 		\Delta(\xi_i)=\sum_{j+k=i}\xi_j^{2^k}\otimes \xi_k\\
	 		\epsilon(\xi_i)=0\text{ , }i\ge 1\\
	 \eta_R(a_{\sigma})=a_{\sigma}\text{ , }\eta_R(a_{\lambda_0})=a_{\lambda_0}\\
	 \eta_R(u_{\sigma})=u_{\sigma}+a_{\sigma}\xi_1\\
	 \eta_R(u_{\lambda_m})=u_{\lambda_m}\text{ , }m>0\\
	 \eta_R(u_{\lambda_0})^{-1}=\sum_ia^{2^i-1}_{\lambda_0}u^{-2^i}_{\lambda_0}\xi_i^2
	 \end{gather}
 \end{prop}
	 \begin{proof}The computation of $(k\wedge k)^{hG}_*=(k\wedge k)^*(BG)$ follows from the computation of $k^{hG}_*=k^*(BG)=k[a]/a^2\otimes k[b]$ and the fact that nonequivariantly, $k\wedge k$ is a free $k$-module. To see that the homotopy fixed point spectral sequence for $k\wedge k$ converges strongly, let $F^iBG$ be the skeletal filtration on the Lens space $BG=S^{\infty}/C_{2^n}$; we can then compute directly that $\lim_i^1(k\wedge k)^*(F^iBG)=\lim_i^1(k[a]/a^2\otimes k[b]/b^i)=0$. 
	 	
	 Thus we get $(k\wedge k)_{\bigstar }^{hG}=k_{\bigstar}^{hG}[\xi_i]$ and the diagonal $\Delta$ and augmentation $\epsilon$ are the same as in the nonequivariant case. The Euler classes $a_{\sigma}, a_{\lambda_0}$ are maps of spheres so they are preserved under $\eta_R$. The action of $\eta_R$ on $u_{\sigma}, u_{\lambda_0}$ can be computed through the right coaction on $k_{\bigstar}^{hG}$: The (completed) coaction of the nonequivariant dual Steenrod algebra on $k^*(BG)=k[a]/a^2\otimes k[b]$ is
	 	\begin{gather}
	 		a\mapsto a\otimes 1\\
	 		b\mapsto \sum_ib^{2^i}\otimes \xi_i^2
	 	\end{gather}
 	To verify the formula for the coaction on $b$ we need to check that $Sq^1(b)=0$ (the alternative is $Sq^1(b)=ab$). From the long exact sequence associated to $0\to \Z/2\to \Z/4\to \Z/2\to 0$, we can see that the vanishing of the Bockstein on $b$ follows from $H^2(C_{2^n};\Z/4)=\Z/4$ ($n>1$).
 	
	 	After identifying $a=a_{\sigma}u_{\sigma}^{-1}$ and $b=a_{\lambda_0}u_{\lambda_0}^{-1}$ we get the formula for $\eta_R(u_{\lambda_0})$ and also that
	 	\begin{equation}
	 		\eta_R(u_{\sigma})=u_{\sigma}+\epsilon a_{\sigma}\xi_1
	 	\end{equation}
	 	where $\epsilon$ is either $0$ or $1$. This is equivalent to $$\eta_R(u_{\sigma}^{-1})=u_{\sigma}^{-1}+\epsilon a_{\sigma}u_{\sigma}^{-2}\xi_1$$ and to see that $\epsilon=1$ we use the map $k^{hC_2}=k^{h(C_{2^n}/C_{2^{n-1}})}\to k^{hC_{2^n}}$ that sends $a_{\sigma},u_{\sigma}$ to $a_{\sigma},u_{\sigma}$ respectively. Finally, to compute $\eta_R(u_{\lambda_m})$ for $m>0$ note that  $$k^{hC_{2^{n-m}}}=k^{hC_{2^n}/C_{2^m}}\to k^{hC_{2^n}}$$ sends $a_{\lambda_0},u_{\lambda_0}$ to $a_{\lambda_m}=0,u_{\lambda_m}$ respectively.
	 \end{proof}
 
 \section{Comparison with Greenlees's description}
 We now compare our result with the description of the Borel Steenrod algebra given in \cite{Gre88}, which is dual to our calculation. 
 
 In our notation, the $G$-spectrum $b$ of \cite{Gre88} is $b=k^h$ and $b^V(X)$ corresponds to $(k^h)^{|V|}_G(X)$; to get $(k^h)^V_G(X)$ we need to multiply with the invertible element $u_V\in k^{hG}_{|V|-V}$. The Borel Steenrod algebra is $b^{\bigstar}_Gb=(k^h)_G^{\bigstar}(k^h)$ and the  Borel dual Steenrod algebra is $b_{\bigstar}^Gb=(k^h)^G_{\bigstar}(k^h)=(k\wedge k)^{hG}_{\bigstar}$.
 
 Greenlees proves that the Borel Steenrod algebra is given by the Massey-Peterson  twisted tensor product (\cite{MP65}) of the nonequivariant Steenrod algebra $k^*k$ and the Borel cohomology of a point $(k^h)_G^{\bigstar}=k^{hG}_{-\bigstar}$. The twisting has to do with the fact that the action of the Borel Steenrod algebra on $x\in (k^h)_G^{\bigstar}(X)$ is given by:
 \begin{equation}
 	(\theta\otimes a)(x)=\theta(ax)
 \end{equation}
where $\theta\in k^*k$ and $a\in k^{hG}_{\bigstar}$. The product of elements $\theta\otimes a$ and $\theta'\otimes a'$ in the Borel Steenrod algebra is not $\theta\theta'\otimes aa'$, since $\theta$ does not commute with cup-products, but rather satisfies the Cartan formula:
\begin{equation}
	\theta(ab)=\sum_i\theta_i'(a)\theta_i''(b)\text{ , }\Delta \theta=\sum_i\theta_i'\otimes \theta''_i
\end{equation}
Therefore:
\begin{equation}
	(\theta\otimes a)(\theta'\otimes a')(x)=\theta(a\theta'(a'x))=\sum_i\theta_i'(a)(\theta''_i\theta')(a'x)
\end{equation}
so
\begin{equation}\label{Product}
	(\theta\otimes a)(\theta'\otimes a')=\sum_i\theta_i'(a)(\theta''_i\theta'\otimes a')
\end{equation}
(we have ignored signs as we are working in characteristic $2$).\\ So the Borel Steenrod algebra is $k^*k\otimes k^{hG}_{\bigstar}$ with twisted algebra structured defined by \eqref{Product}.

Moreover, Greenlees expresses the action of $k^*k$ on $(k^h)^{\bigstar}_G(X)$ in terms of the action of $k^*k$ on the orientation classes $u_V$ and the usual (nonequivariant) action of $k^*k$ on $(k^h)^*_G(X)=k^*(X\wedge_GEG_+)$. This is done through the Cartan formula: If $x\in (k^h)^V_G(X)$ then $u_V^{-1}x\in (k^h)^{|V|}_G(X)$ and
\begin{equation}
	\theta(x)=\theta(u_Vu_V^{-1}x)=\sum_i\theta_i'(u_V)\theta''_i(u_V^{-1}x)
\end{equation}
What remains to compute is $\theta_i'(u_V)$, namely the action of $k^*k$ on orientation classes. 

In our case, for $G=C_{2^n}$, we can see that:
\begin{prop}The action of $k^*k$ on orientation classes is determined by:
\begin{gather}
	Sq^i(u_{\sigma})=\begin{cases}u_{\sigma}&i=0\\
		a_{\sigma}&i=1\\
		0&\text{ otherwise}
	\end{cases}\\
	Sq^i(u_{\lambda_m})=\begin{cases}u_{\lambda_m}&i=0\\
		a_{\lambda_0}&i=2, m=0\\
		0&\text{ otherwise}
		\end{cases}
\end{gather}
\end{prop}
\begin{proof} Compare with the proof of Proposition \ref{Power2Borel}.\end{proof}
The twisting in the case of the Borel dual Steenrod algebra corresponds to the fact that $(k\wedge k)^{hG}_{\bigstar}$ is a Hopf algebroid and not a Hopf algebra; computationally this amounts to the formula for $\eta_R$ of Proposition \ref{Power2Borel}. 
	
\phantom{1}\smallbreak

\begin{small}
	\noindent  \textsc{Department of Mathematics, University of Chicago}\\
	\textit{E-mail:} \verb|nickg@math.uchicago.edu|\\
	\textit{Website:} \href{http:://math.uchicago.edu/~nickg}{math.uchicago.edu/$\sim$nickg}
\end{small}

\end{document}